\documentclass[11pt,a4paper]{article}

\usepackage{amsmath}
\usepackage{amssymb}
\usepackage{amsthm}
\usepackage{verbatim}
\usepackage{graphicx}
\usepackage{color}
\usepackage{bm}

\makeatletter

   \@addtoreset{equation}{section}
\makeatother

\newcommand{\be}{\begin{equation}}
\newcommand{\ee}{\end{equation}}
\newcommand{\ba}{\begin{eqnarray}}
\newcommand{\ea}{\end{eqnarray}}
\newcommand\nn{\nonumber}

\newcommand{\cA}{{\cal A}}
\newcommand{\cB}{{\cal B}}

\newcommand{\cO}{{\cal O}}
\newcommand{\cM}{{\cal M}}

\newcommand{\us}{\rule{0.2cm}{0.4pt}}

\def\a{\alpha}
\def\e{\epsilon}

\def\la{\lambda}

\theoremstyle{definition}
\newtheorem{defn}{Definition}[section]
\theoremstyle{theorem}
\newtheorem{thm}{Theorem}[section]
\newtheorem{propo}{Proposition}[section]

\begin{document}

\vskip 12mm

\begin{center} 
{\Large \bf 
Motzkin path on RNA abstract shapes}
\vskip 10mm
{ \large  Sang Kwan Choi\footnote{email: hermit1231@sogang.ac.kr; skchoi@scu.edu.cn} }\\
\vskip 5mm
{\it   Center for Theoretical Physics, College of Physical Science and Technology\\
Sichuan University, Chengdu 610064, China} 
\end{center}

\vskip 5mm

\begin{abstract}
We consider a certain abstract
of RNA secondary structures,
which is closely related to RNA shapes.
The generating function 
counting the number of the abstract structures
is obtained by means of Narayana numbers
and 2-Motzkin paths, through which
we provide an identity 
related to Narayana numbers and
Motzkin polynomials.
Furthermore, 
we show that a combinatorial interpretation 
on 2-Motzkin paths 
leads to the correspondence
between 1-Motkzin paths and RNA shapes,
which facilitates probing
further classifications 
or abstractions of the shapes.
In this paper,
we classify the shapes 
with respect to the number of components
and calculate their asymptotic distributions.
\vskip 2mm
\noindent
{\it {\bf Keywords: }RNA shape; Motzkin path; Narayana number; Asymptotics}
\end{abstract}

\vskip 5mm

\setcounter{footnote}{0}

\section{Introduction} 

Ribonucleic acid (RNA) is a single stranded molecule
with a backbone of nucleotides,
each of which has one of the four bases,
adenine (A), cytosine (C), guanine (G) and uracil (U).
Base pairs are formed intra-molecularly 
between A-U, G-C or G-U, 
leading the sequence of bases to form helical regions. 
The primary structure of an RNA is merely the sequence of bases
and its three-dimensional conformation by base pairs
is called the tertiary structure.
As an intermediate structure between the primary and the tertiary,
the secondary structure is a planar structure
allowing only nested base pairs.
This is easy to see in its diagrammatic representation, see Fig.\ref{fig:fig1}.
A sequence of $n$ bases is that of labeled vertices $(1,2,\cdots, n)$
in a horizontal line 
and base pairs are drawn as arcs in the upper half-plane.
The condition of nested base pairs means
non-crossing arcs: for two arcs (i,j) and (k,l)
where $i < j$, $k<l$ and $i<k$, either $i<j<k<l$ or $i<k<l<j$.
Since the functional role of an RNA depends mainly
on its 3D conformation,
prediction of RNA folding
from the primary structure
has long been an important problem in molecular biology. 
The most common approach for the prediction
is free energy minimization
and many algorithms to compute the structures
with minimum free energy
have been developed (see for instance,
\cite{NS_1980, ZS_1981, ZS_1984, SSR_1997}).

\begin{figure}[!tpb]
\centering
\includegraphics[width=0.8\textwidth]{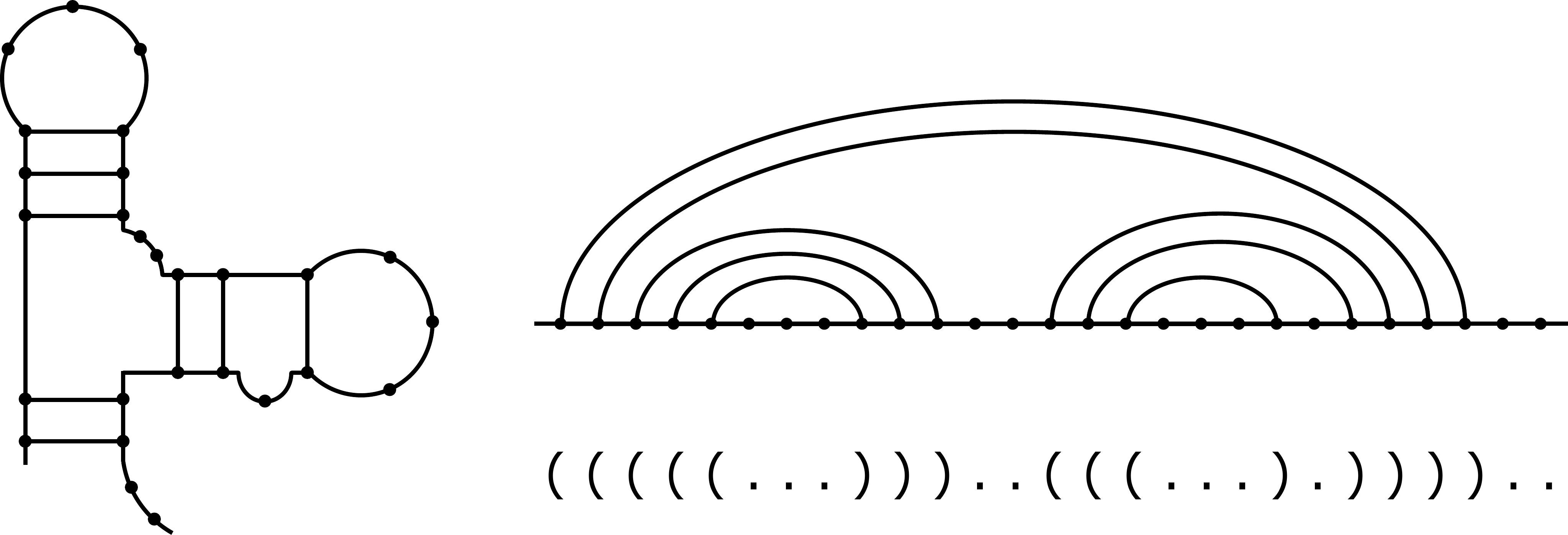}
\caption{\label{fig:fig1}
Representations of secondary structures.
The RNA structure on the left hand side
is represented as the diagram (top right)
and the dot-bracket string (bottom right).
}
\end{figure}

On the other hand,
RNA structures are often considered
as combinatorial objects
in terms of representations 
such as strings over finite alphabets,
linear trees or the diagrams.
Combinatorial approaches
enumerate the number of possible
structures under various kinds of constraints
and observe its statistics
to compare with experimental findings 
\cite{SW_1979, HSS_1998, SW_1994, BLR_2016, CRU}.
They also provide classifications 
of structures to advance prediction algorithms
\cite{W_1978, GVR_2004, OZ_2002, RHAPSN_2011}.

In this paper, 
we consider a certain abstract 
of secondary structures 
under a combinatorial point of view
regardless of primary structures.
The abstract structure is, in fact, 
closely related 
to so-called RNA shapes 
\cite{GVR_2004, JRG_2008, NS_2009}, see section \ref{sec:shape}.
Although we will consider it 
apart from prediction algorithms,
let us briefly review  the
background to RNA shapes
in the context of prediction problem.
In free energy minimization scheme,
the lowest free energy structures
are not necessarily native structures.
One needs to search 
suboptimal foldings in a certain energy bandwidth
and, in general, obtains a huge set of suboptimal foldings.
RNA shapes classify
the foldings according to their structural similarities
and provide so-called shape representatives
such that native structures
can be found among those shape representatives.
Consequently, 
it can greatly narrow down the huge set of suboptimal foldings
to probe in order to find native structures.

In the following preliminary,
we introduce our combinatorial object,
what we call island diagram
and present basic definitions
needed to describe the diagram.
In section \ref{sec:gf}, we find the generating function 
counting the number of 
island diagrams in two different ways
and through which, one may see the relation between 
Narayana numbers and 2-Motzkin paths.
In particular, we find a combinatorial identity,
see equation \eqref{ouriden},
which generalizes the following two identities 
that Coker provided \cite{Coker} 
(see also \cite{CYY_2008} for a combinatorial interpretation):
\begin{gather}
\sum_{k=1}^{n}\frac1n \binom{n}{k} \binom{n}{k-1} x^{k-1}
=\sum_{k=0}^{\lfloor \frac{n-1}2 \rfloor}
C_k \binom{n-1}{2k} x^k (1+x)^{n-2k-1}
\label{coker1}
\\
\sum_{k=1}^n \frac1n \binom{n}{k} \binom{n}{k-1} 
x^{2(k-1)}(1+x)^{2(n-k)}
=\sum_{k=1}^n C_k \binom{n-1}{k-1}x^{k-1} (1+x)^{k-1}
\label{coker2}
\end{gather}
where $C_k$ is the Catalan number defined by
$C_k=\frac1{k+1}\binom{2k}{k}$ for $k \geq 0$.
A combinatorial interpretation on 2-Motzkin paths
is given in accordance with island diagrams.
The interpretation implies the bijection between
$\pi$-shapes and 1-Motzkin paths
which was shown in \cite{DS_1977, LPC_2008}.
The refined bijection map facilitates
exploring further
classifications or abstractions of $\pi$-shapes.
As one immediate attempt, in section \ref{sec:shape},
we classify $\pi$-shapes according to the number of components,
of which the generating function is calculated.
Asymptotic distributions of the number of $\pi$-shapes
are presented as a function of the number of components.
Section 1 and Section 2 overlap in part with \cite{CRU_2019}.

\subsubsection*{Preliminary}
A formal definition of secondary structures is given as follows:
\begin{defn}[Waterman \cite{W_1978}]
A secondary structure is a vertex-labeled graph on $n$ vertices
with an adjacency matrix $A=(a_{ij})$ 
(whose element $a_{ij}=1$ if $i$ and $j$ are adjacent,
and $a_{ij}=0$ otherwise with $a_{ii}=0$)
fulfilling the following three conditions:
\\
1. $a_{i,i+1}=1$ for $1 \leq i \leq n-1$.
\\
2. For each fixed $i$, there is at most one $a_{ij}=1$
where $j \neq i \pm 1$
\\
3. If $a_{ij}=a_{kl}=1$, where $i <k<j$, then $i \leq l \leq j$.
\label{waterman}
\end{defn}
\noindent An edge $(i,j)$ with $|i-j| \neq 1$ is said to be a base pair
and a vertex $i$ connected only to $i-1$ and $i+1$ is called unpaired.
We will call an edge $(i, i+1)$, $1 \leq i \leq n-1$, a backbone edge.
Note that a base pair between adjacent two vertices
is not allowed by definition
and the second condition implies
non-existence of base triples.

There are many other representations of secondary structures
than the diagrammatic representation.
In this paper, we often use the so-called dot-bracket representation,
see figure \ref{fig:fig1}.
A secondary structure can be represented as
a string $\mathbf{S}$ over the alphabet set $\{(,\, ),\, .\}$
by the following rules \cite{HSS_1998}:
\\
1. If vertex $i$ is unpaired then $\mathbf S_i=$``$.$''.
\\
2. If $(i,j)$ is a base pair and $i<j$ then $\mathbf S_i=$ ``('' 
and $\mathbf S_j=$``)''.

In the following, we present the basic definitions 
of structure elements needed for our investigations.
\begin{defn} 
A secondary structure on $(1, 2, \cdots, n)$ 
consists of the following structure elements (cf. Fig.\ref{fig:fig2}).
By a base pair $(i,j)$, we always assume $i<j$.
\\
1. The sequence of unpaired vertices $(i+1, i+2, \cdots, j-1)$
is a {\it hairpin} if $(i,j)$ is a base pair.
The pair $(i,j)$ is said to be the 
{\it foundation of the hairpin}.
\\
2. The sequence of unpaired vertices $(i+1, i+2, \cdots, j-1)$
is a {\it bulge} if either $(k,j)$, $(k+1,i)$ 
or $(i,k+1)$, $(j,k)$ are base pairs. 
\\
3. A {\it tail} is a sequence of unpaired vertices
$(1,2, \cdots, i-1)$, resp. $(j+1, j+2, \cdots, n)$
such that $i$, resp. $j$ is paired.
\\
4. An {\it interior loop} is two sequences of unpaired vertices
$(i+1, i+2, \cdots, j-1)$ and $(k+1, k+2, \cdots, l-1)$
such that $(i,l)$ and $(j,k)$ are pairs, where $i<j<k<l$.
\\
5. For any $k \geq 3$ and $0 \leq l, m \leq k$ with $l+m=k$,  
a {\it multi loop}
is $l$ sequences of unpaired vertices
and $m$ empty sequences
$(i_1+1, \cdots, j_1-1), (i_2+1, \cdots, j_2-1), 
\cdots, (i_k+1, \cdots, j_k-1)$
such that $(i_1, j_k), (j_1, i_2), \cdots, (j_{k-1}, i_k)$
are base pairs.
A sequence $(i+1, \cdots, j-1)$ 
is an empty sequence if $i+1 = j$.
\\
6. For any $k \geq 2$ and $0 \leq l, m \leq k$ with $l+m=k$,  
an {\it external loop}
is $l$ sequences of unpaired vertices
and $m$ empty sequences
$(1, \cdots, j_1-1), (i_2+1, \cdots, j_2-1), 
\cdots, (i_k+1, \cdots, n)$
such that $(j_1, i_2), \cdots, (j_{k-1}, i_k)$
are base pairs.
The number of {\it components}
is $k-1$.
\\
7. A {\it stack (or stem)} consists of uninterrupted
base pairs $(i+1,j-1), (i+2,j-2), \cdots, (i+k, j-k)$
such that neither $(i,j)$ nor $(i+k+1, j-k-1)$
is a base pair. Here the {\it length} of the stack is $k$.
\label{elem}
\end{defn}

\begin{figure}[hb]
\centering
\includegraphics[width=0.9\textwidth]{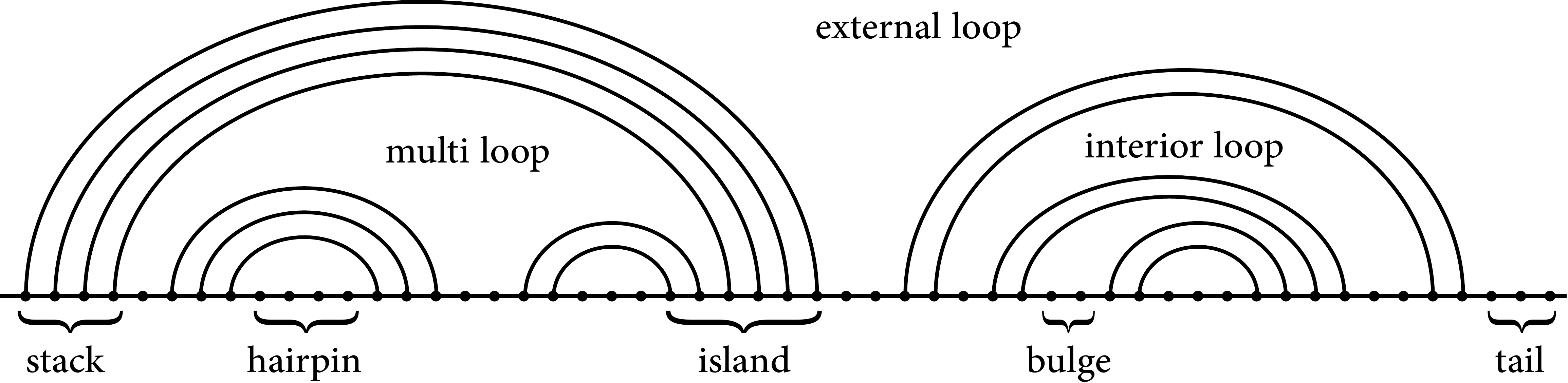}
\caption{\label{fig:fig2}
Structure elements of secondary structures.
}
\end{figure}

\noindent Note that, while other structure elements
consist of at least one vertex,
a multiloop and an external loop 
do not necessarily have a vertex.
In the diagrammatic representation, 
a multiloop is a structure bounded
by three or more base pairs and backbone edges.

In addition to the structure elements,
we define an auxiliary element indicating
a sequence of maximally consecutive paired vertices:
\begin{defn}
An {\it island} is a sequence of paired vertices
$(i,i+1, \cdots, j)$ such that
\\
1.  $i-1$ and $j+1$ are both unpaired, where $1<i \leq j<n$.
\\
2. $j+1$ is unpaired, where $i=1$ and $1 \leq j<n$.
\\
3. $i-1$ is unpaired, where $1<i \leq n$ and $j=n$.
\end{defn}

Now we introduce the abstract structures
to consider in the next section.
From here on, we will call
the structures {\it island diagrams}
for convenience.
An island diagram (cf. Fig.\ref{fig:island}) 
is obtained from secondary structures by
\\
1. Removing tails.
\\
2. Representing a sequence of consecutive unpaired vertices
between two islands by a single blank.
\\
Accordingly, we retain unpaired regions except for tails
but do not account for the number of unpaired vertices.
In terms of the dot-bracket representation,
we shall use the underscore ``\us'' for the blank:
for example, the island diagram ``$((\us)\us)$'' abstracts 
the secondary structure ``$((...)....)$''.
Since the abstraction preserves
all the structure elements (except for tails) 
in the definition \ref{elem},
we will use them to describe island diagrams
in such a way that, for instance, the blank is a hairpin
if its left and right vertices are paired to each other.

\begin{figure}[hb]
\centering
\includegraphics[width=0.8\textwidth]{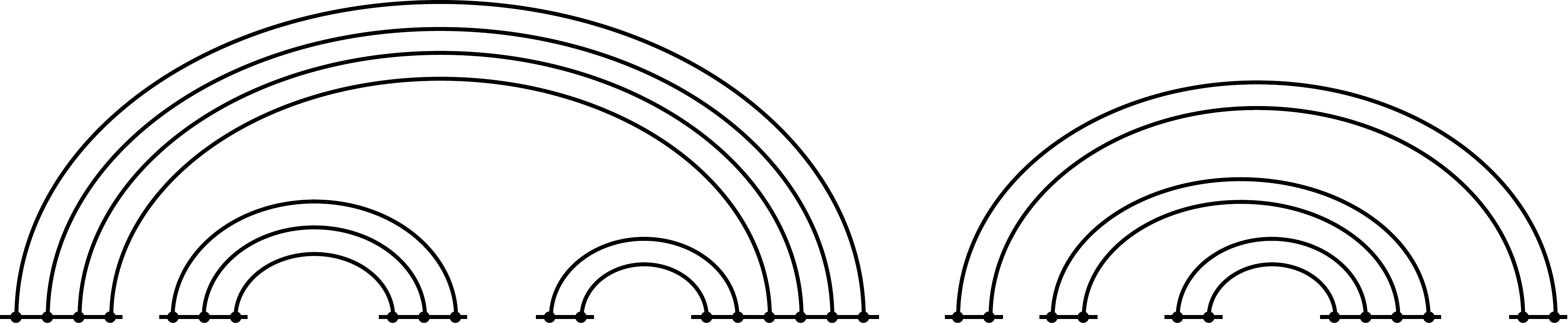}
\caption{\label{fig:island}
An example of island diagrams.
This island diagram is the abstract structure
of the secondary structure given in the figure \ref{fig:fig2}.
}
\end{figure}

\section{Generating function for island diagram}\label{sec:gf}
We enumerate the number of island diagrams
$g(h,I,\ell)$,
filtered by the number of hairpins($h$), islands($I$) and basepairs($\ell$).
Let $G(x,y,z)=\sum_{h,I,\ell} g(h,I,\ell) x^h\, y^I \, z^\ell$
denote the corresponding generating function.
We obtain the generating function in two different ways,
by means of Narayana numbers and 2-Motzkin paths.
In particular, we provide a bijection map
between 2-Motzkin paths 
and sequences of matching brackets.

\subsection{Narayana number}
The easiest way to obtain the generating function $G(x,y,z)$
is to use a combinatorial interpretation of the Narayana numbers,
which are defined by
\be
N(n,k)=\frac1{n}\binom{n}{k}\binom{n}{k-1}\,, ~~~
1 \leq k \leq n  \,.
\ee
The Narayana number $N(n,k)$ counts
the number of ways arranging $n$ pairs of brackets
to be correctly matched and contain $k$ pairs as ``()''.
For instance, the bracket representations for $N(4,2)=6$
are given as follows:
\begin{center}
(()(())) ~~~~ ((()())) ~~~~ ((())()) ~~~~
()((())) ~~~~ (())(()) ~~~~  ((()))()
\end{center}
It is easy to recover island diagrams from this representation.
 
\begin{propo}
The generating function has the form
\be
G(x,y,z)= \sum_{\ell, h} N(\ell,h)\, x^h\, y^{h+1} \, 
(1+y)^{2\ell-1-h} \,z^{\ell} \,.
\label{GFN}
\ee
Its closed form is
\be
G(x,y,z)=\left(\frac{y}{1+y} \right)
\frac{1-A(1+B)-\sqrt{1-2A(1+B)+A^2 (1-B)^2}}{2A}
\label{cGFN}
\ee
where $A=z(1+y)^2$ and $B=x \, y /(1+y)$.
\end{propo}
 
\begin{proof}
One may immediately associate bracket representations
of the Narayana numbers with island diagrams.
Without regard to underscores,
the pair of brackets is associated with the basepair
and the sub-pattern ``()'' corresponds to the
foundation of the hairpin. 
It clearly explains the factor $N(\ell,h) x^h z^\ell$.
Now we consider the insertions of underscores 
to recover the string representation of island diagrams.
Recall that, in secondary structures, 
a hairpin consists of at least one unpaired vertices.
Therefore, the foundation of the hairpin
 ``()'' must contain a underscore ``(\us)''.
The number $h$ of underscores are so inserted
that we have the factor $y^{h+1}$.
After the insertion of hairpin underscores, 
there are $(2\ell-1-h)$ places 
left to possibly insert underscores.
The numbers of all possible insertions 
are summarized by the factor $(1+y)^{2\ell-1-h}$.
Therefore, one obtains the form \eqref{GFN}.
The generating function of 
the Narayana numbers is well-known (see for instance \cite{BH_2011})
so that one writes the closed form \eqref{cGFN}.
\end{proof}

\subsection{2-Motzkin path}
The generating function $G(x,y,z)$ can also be written 
in terms of Motzkin polynomial coefficients.
The Motzkin numbers $M_n$ 
and the Motzkin polynomial coefficients $M(n,k)$ 
are defined as
\be
M_n=\sum_{k=0}^{\lfloor n/2 \rfloor} M(n,k)
~~~ \mathrm{where} ~~~
M(n,k)=\binom{n}{2k} C_k \,.
\label{motzkindef}
\ee
Let us consider the combinatorial identity
in the following theorem.
It is easy to prove using the generating function
of the Motzkin polynomials:
\be
m(v,w):=\sum_{n \geq 0}\sum_{k=0}^{\lfloor n/2 \rfloor}
M(n,k) w^{n} v^k=\frac{1-w-\sqrt{(1-w)^2-4 v  w^2}}{2v  w^2}\,.
\label{gfmotzkin}
\ee 
\begin{thm} For any integer $\ell \geq 1$, there holds
\be
\begin{split}
&\frac{y}{1+y}\, \sum_{h=1}^{\ell} N(\ell,h) \, (x\,y)^h \,(1+y)^{2\ell-h}
\\
&= x \, y^2 \sum_{p=0}^{\lfloor \frac{\ell-1}2 \rfloor}
M(\ell-1,p)\, \big(x \, y \,(1+y)^3 \big)^p  \,
\big((1+y)(1+y+x \,y)\big)^{\ell-2p-1} \,.
\end{split}
\label{ouriden}
\ee
\end{thm}
\begin{proof}
The left hand side
is $[z^\ell] G(x,y,z)$ given in \eqref{GFN}.
Multiplying $z^\ell$ and taking the summation over $\ell$ at each side,
one can check that the right hand side
is indeed the generating function $G(x,y,z)$.
\end{proof}

\noindent Note that the identity \eqref{ouriden} reproduces
the Coker's two identities.
When we substitute $x/y$ for $x$ and then put $y=0$,
we get the identity \eqref{coker1}.
Furthermore, the substitution $x \to y/(1+y)$
leads to the identity \eqref{coker2}.\footnote{
In order to deduce the identity,
one may need the Touchard's identity \cite{Touchard}:
$C_n=\sum_k C_k \binom{n-1}{2k} 2^{n-2k-1}$,
which can also be derived from \eqref{coker1}
when $x=1$.}

We will investigate how the right hand side in \eqref{ouriden}
represents island diagrams.
In order to do that, 
we need a combinatorial interpretation
of 2-Motkzin paths.
Let us first introduce the Motzkin paths,
that can also be called 1-Motkzin paths.
A Motzkin path of size $n$ is a lattice path 
starting at $(0,0)$ and ending at $(n,0)$
in the integer plane $\mathbb Z \times \mathbb Z$,
which satisfies two conditions: (i) It never passes below the $x$-axis. 
(ii) Its allowed steps are the up step $(1,1)$,
the down step $(1,-1)$ and the horizontal step $(1,0)$.
We denote by $U$, $D$ and $H$
an up step, a down step and a horizontal step, respectively. 
The Motzkin polynomial coefficient $M(n,k)$
is the number of Motzkin paths of size $n$ with $k$ up steps.
Since the Motkzin number $M_n$
is given by the sum of $M(n,k)$ over 
the number of up steps,
$M_n$ is the number of Motzkin paths of size $n$.
See for instance, the following figure depicting a Motzkin path of $M(7,2)$:

\begin{figure}[ht]
\centering
\includegraphics[width=0.4\textwidth]{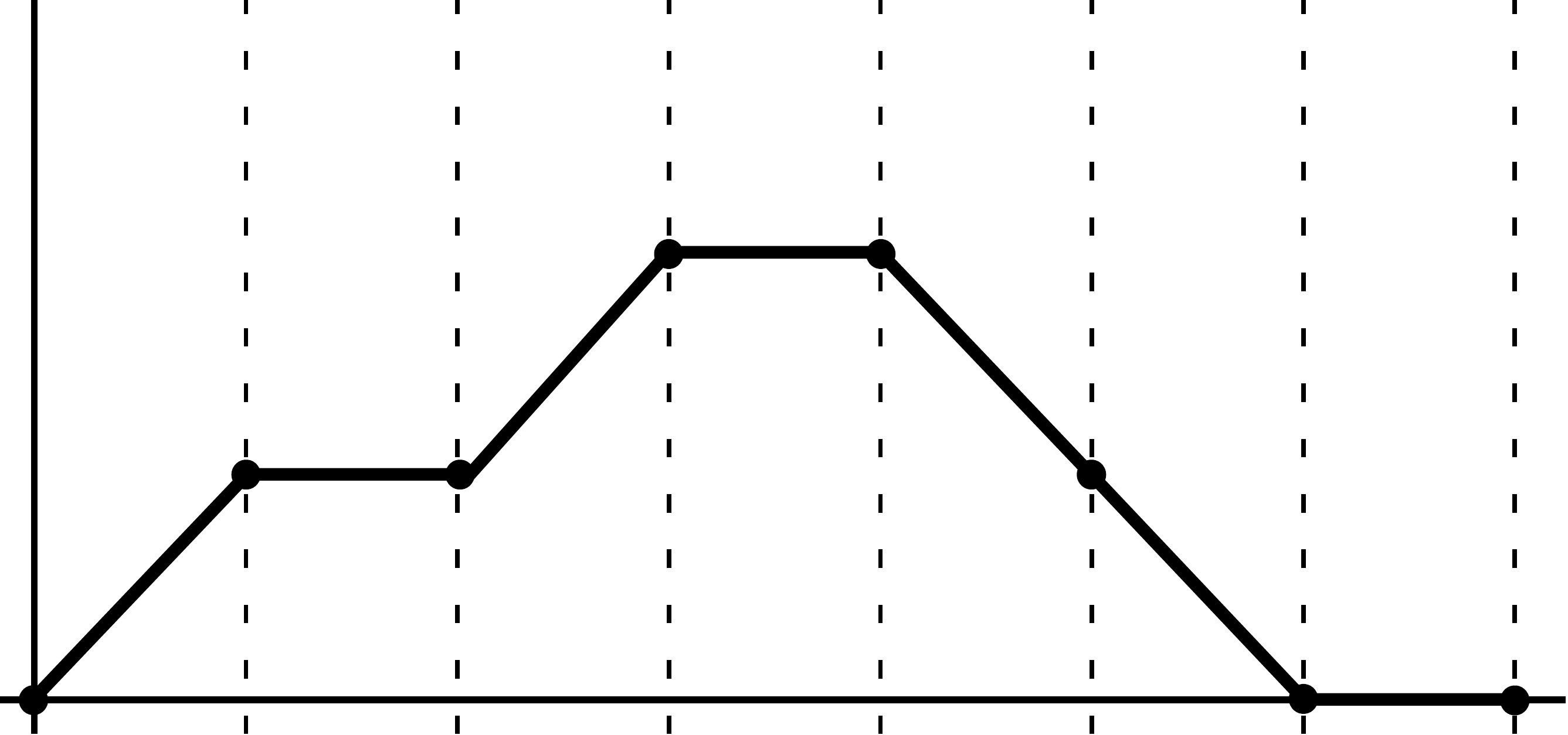}
\caption{Motzkin path of $UHUHDDH$.}
\end{figure}  

On the other hand, 2-Motzkin paths
allow two kinds of horizontal steps,
which often distinguish one from another by a color,
let us say, $R$ and $B$
denoting a red and a blue step, respectively.
We provide a bijection map between 2-Motzkin paths
and strings of matching brackets.\footnote{
Sequences of matching brackets 
are only Dyck paths. A bijection map
between Dyck paths and 2-Motzkin paths was introduced
by Delest and Viennot \cite{DV_1984}.
But here we present a different way
of mapping than the well-known one.}
Suppose we have a 2-Motzkin path of size $n$
given by a string $q_1 \, q_2 \, \cdots q_n$ over the set $\{U,D,R,B\}$.
The corresponding string of brackets $S_n$
can be obtained by the following rules: 
\\
(i) We begin with ``()'' : Let $S_0=()$.
\\
(ii) For any $1 \leq k \leq n$, suppose
there exist a string of brackets $S'$ and 
a string of matching brackets $S''$
which are possibly empty
such that $S_{k-1}$ has the form $S'(S'')$.
Then $S_k$ is given by
\begin{center}
$S'((S'')()$ \, if $q_k=U$,
\qquad
$S'(S''))$ \, if $q_k=D$,
\\
$S'(S'')()$ \, if $q_k=R$,
\qquad
$S'((S''))$\, if $q_k=B$.
\end{center}
For example, the string of matching brackets
corresponding to the 2-Motzkin path $UBURDD$
is obtained as follows:
\be
() \stackrel{U}{\longrightarrow} (()()
\stackrel{B}{\longrightarrow} (()(())
\stackrel{U}{\longrightarrow} (()((())()
\stackrel{R}{\longrightarrow} (()((())()()
\stackrel{D}{\longrightarrow} (()((())()())
\stackrel{D}{\longrightarrow} (()((())()()))
\nn
\ee

We remark here that only blue steps
can make a stack. In other words,
directly nested structures such as ``$(())$''
never occur without blue steps.
Therefore,
a 1-Motzkin path can be translated
into a string of matching brackets without
directly nested brackets.
This is one of 
the 14 interpretations of Motzkin numbers
provided by Donaghey and Shapiro in \cite{DS_1977}.
Later, in \cite{LPC_2008}, 
it was also shown using context-free grammars
in the context of RNA shapes.
We also remark that the Motzkin polynomial coefficient $M(\ell-1, u)$
is the number of ways arranging $\ell$ pairs of brackets
to be correctly matched
and contain $\ell-u$ pairs as ``$()$''
with no occurrence of directly nested bracket.

Now we go back to the generating function
on the right hand side in \eqref{ouriden}
and rewrite it as
\be
\begin{split}
G(x,y,z)=&\sum_{\ell, u} M(\ell-1,u) \,
(x y^2 z)\,  \big((1+y) \sqrt z\big)^u \, \big( x y (1+y)z \big)^u
\\
&~~~\times \big((1+y) \sqrt z\big)^d \,
\big((1+y)^2z + xy(1+y)z \big)^s
\end{split}
\label{motzkinGFN}
\ee
where $u$, $d$ and $s$ stand for
the number of up, down and horizontal steps, respectively
($u=d$, $u+d+s=\ell-1$).
Let us explain each factor 
in detail by means of the above rules.
The term $xy^2 z$ is merely the starting hairpin ``(\us)''
(recall that the exponent of $x$, $y$ and $z$
are the number of hairpins, islands and basepairs, resp.).
At each up step, one has a left bracket
and a hairpin to add.
For a given non-empty string $S$ of island diagrams,
suppose that we add a left bracket
then there are the two possibilities, ``$(S\,$'' and ``$(\us S\,$''
corresponding to $\sqrt z$ and $y \sqrt z$, respectively.
Thus, we get the factor $(1+y) \sqrt z$
at every up step and, in the same manner, at every down step.
Likewise, adding a hairpin introduces the factor $xy(1+y)z$
since ``$S(\us)$'' and ``$S\us(\us)$''
corresponds to $x y z$ and $x y^2 z$, respectively.
On the other hand, 
a horizontal step can be either $R$ or $B$.
A red step is to add a hairpin and corresponds to $xy(1+y)z$.
A blue step  
is to add one basepair nesting the string ``$(S)$''
and there are three possibilities:
the stack ``$((S))$'' for z, the two bulges ``$(\us(S))$'' 
and ``$((S)\us)$'' for $yz$
and the interior loop ``$(\us(S)\us)$'' for $y^2 z$.
Therefore, we get
$((1+y)^2 z+ xy(1+y)z)$ at each horizontal step.

\section{Motzkin path and RNA shape}
\label{sec:shape}
The bijection map given in the previous section
reduces to the correspondence between 
1-Motkzin paths and RNA shapes, 
which is called $\pi$-shape.
In this section,
we exploit 1-Motzkin paths
to classify $\pi$-shapes, and
their generating functions
are given to calculate asymptotics.
By Motzkin paths, from here on,
we mean 1-Motzkin paths.

\subsection{Motzkin path and $\pi$-shape}

The Motzkin path is closely related to
the $\pi$-shape (or type 5),
which is one of the five RNA abstract shapes 
provided in \cite{GVR_2004}
classifying secondary structures
according to their structural similarities.
In order to describe $\pi$-shape,
let us first consider $\pi'$-shape (or type 1).
$\pi'$-shape
is an abstraction of secondary structures
preserving their loop configurations and unpaired regions.
A stem is represented as one basepair
and a sequence of maximally consecutive unpaired vertices
is considered as an unpaired region 
regardless of the number of unpaired vertices in it.
In terms of the dot-bracket representation,
a length $k$ stem ``$(^k \cdots )^k$'' is represented by
a pair of squared brackets ``$\bm{[} \cdots \bm]$''
and an unpaired region is depicted by an underscore. 
For instance, the $\pi'$-shape ``
$\us\bm[\bm[\bm[\us\bm]\us\bm[\us\bm]\bm]\us\bm]$''
abstracts the secondary structure ``
$...((((...)..((...))))..)$''.

In addition to the abstraction of $\pi'$-shape,
$\pi$-shape ignores unpaired regions.
Removing unpaired vertices,
two base pairs encircling
a bulge or an interior loop become consecutive base pairs
which then merge into a square bracket
according to the abstraction of a stem.
For example, the $\pi'$-shape
``$\us\bm[\bm[\bm[\us\bm]\us\bm[\us\bm]\bm]\us\bm]$'' 
results in the $\pi$-shape ``$\bm[ \bm[\bm]\bm[\bm]\bm]$''.
Consequently,
$\pi$-shapes 
retain only hairpin and multiloop configurations.
One may immediately notice that
the string representations of $\pi$-shapes
are nothing but the sequences of matching brackets without
directly nested brackets.
Therefore, as is remarked in the previous section,
we establish the bijection map between $\pi$-shapes 
and 1-Motzkin paths.
Recalling that a red step adds a hairpin
and an up step creates a multiloop and a hairpin,
the Motzkin polynomial coefficient $M(\ell-1,u)$
is the number of $\pi$-shapes
with $u$ multiloops and $\ell-u$ hairpins.

It is obivous that
$\pi$-shapes can be classified
with respect to the number of multiloops and hairpins
through the bijection map to Motzkin paths.
The significance of $\pi$-shapes, on the other hand,
is that the abstraction enables us to focus on the branching pattern
of folding.
Although each $\pi$-shape has different branching patterns,
one may still group them
according to a branching pattern similarity.
As one attempt to take account of such similarities,
we consider here the number of components.
Namely, we classify $\pi$-shapes
by the number of branches of the external loop.
Note that the number of components
is equivalent to $(r_0-1)$ in Motzkin path
where $r_0$ denotes the number of horizontal steps at level 0.
Therefore, in other words, we classify Motzkin paths by $r_0$.

\begin{table}[ht]
\centering
\begin{tabular}{|c c l|}
\hline 
~ 1-Motzkin path~  &     & ~~~~~~  $\pi$-shape \\  \hline \hline
 $u$   & $\longleftrightarrow$  & number of multi loops \\ \hline
$u+r+1$ & $\longleftrightarrow$   & number of hairpin loops \\ \hline
 $r_0+1$ & $\longleftrightarrow$   &  number of components 
  \\  \hline
\end{tabular}
\caption{Relation between Motkzin path and $\pi$-shape.
$u$, $r$ and $r_0$ denote
the number of up steps, horizontal steps
and horizontal steps at level 0, respectively.}
\label{tab:1}
\end{table}

\begin{propo}
The number of Motzkin paths of size $n$
with $r_0$ horizontal steps at level 0
is given by
\be
M(r_0;n)=\sum_{u=1}^{\lfloor \frac{n-r_0}2 \rfloor} M(r_0;n,u)
 ~~~ if ~~ n \neq r_0
\label{MN}
\ee
with $M(n;n)=1$
where
\be
M(r_0;n,u) = \frac{r_0+1}{n+1}
\binom{n+1}{u}F(n-r_0-1 \,, u-1) 
\ee
and $F(a,b)=\binom{a-b-1}{b}$ 
is Fibonacci polynomial coefficient.
\label{prop:mot}
\end{propo}
\begin{proof}
See Appendix \ref{app:1}.
\end{proof}
With a fixed size $n$, one may plot the distribution 
of $M(r_0;n)$ as a function of $r_0$
to find that the number of Motzkin paths
decreases as $r_0$ increases for $r_0 \geq 1$.
See Figure \ref{mplot}.
One interesting feature is that
$M(0;n)$ and $M(1;n)$ differ by $(-1)^n$,
which is not straightforward to derive
for arbitrary $n$ using the explicit form \eqref{MN}.
In order to prove the feature as well as 
to investigate the asymptotic distribution at large $n$,
one needs to find its generating function.

Let $R_0$ stand for a horizontal step at level 0.
A brief description of the method is as follows: first we construct
the generating function for the Motzkin paths without $R_0$
and then, put those paths together with $R_0$ steps
to obtain the desired generating function.
For convenience,
let us denote by $\{\epsilon\}, \cM$ and $\cA$
a path with size zero which means identity,
the class of all Motzkin paths
and the class of the Motzkin paths without $R_0$, respectively.
We begin with the generating function $m(v,w)$ 
of $\cM$ given in \eqref{gfmotzkin}.
By adding an up step($U$) and a down step($D$)
at each end of $\cM$,
one obtains the class 
$\widehat{\cM}$ of the Motzkin paths that never touch the level 0
apart from the starting and ending points.
In terms of the symbolic enumeration methods,
this can be written as $\widehat{\cM} = U \times \cM \times D$.
Recalling that $v$ and $w$ in $m(v,w)$
are the expansion variables for up-step and size,
$U \times \cM \times D$ translates into
the generating function $v w^2 m(v,w)$.
The paths in $\cA$ can then be achieved
by concatenating the paths in $\widehat{\cM}$ 
that is symbolically represented as
\be
\cA = SEQ(\widehat{\cM}) :=
\{\epsilon\}+\widehat{\cM}+(\widehat{\cM} \times \widehat{\cM})+
(\widehat{\cM} \times \widehat{\cM} \times \widehat{\cM}) 
+ \cdots 
\ee
such that one obtains the generating function $A(v,w)$ of $\cA$ as
\be
A(v,w)=\frac1{1-v w^2 m(v,w)} \,.
\ee
Now we glue the paths in $\cA$ together with $R_0$ steps.
With a given number $k$ of $R_0$ steps,
there are $(k+1)$ empty slots to be filled with
either $\{\epsilon\}$ or $(\cA \setminus\{\epsilon\})$.
For instance,
the Motzkin paths with $r_0=1$ is represented by
\be
(\{\e\} \times R_0 \times \{\e\})
+ (\{\e\} \times R_0 \times \cB)
+ (\cB \times R_0 \times \{\e\})
+(\cB \times R_0 \times \cB)
\ee
where $\cB=(\cA \setminus \{\e\})$
that yields $t w A(v,w)^2$
when we employ $t$ as the expansion variable for $r_0$.
Likewise, when $r_0=k$, 
one finds $t^k w^k A(v,w)^{k+1}$.
Therefore, the desired generating function is 
\be
m(t;v,w) := \sum_{r_0,n,u}M(r_0;n,u) \,t^{r_0} \, v^u \, w^n
= \frac{A(v,w)}{1-t w A(v,w)} \,.
\label{cMN}
\ee
It is now straightforward to prove the relation
\be
M(0;n)-M(1;n)=(-1)^n ~~~\mathrm{for}~~ n \geq 1 
\ee
since $ M(0;n)-M(1;n)=[w^n](A(1,w)-w A (1,w)^2)$
and $A(1,w)-w A(1,w)^2=1/(1+z)$.
Furthermore,
one can exploit the generating function $m(t;1,w)$
to find the following asymptotics.
See Appendix \ref{app:2-1} for its proof.
\begin{thm} 
Let $r_0$ be a non-negative integer.
For $r_0 \ll n$,  the distribution of 
$M(r_0;n)$ holds
\be
\lim_{n \to \infty}
\frac{M(r_0;n)}{M_n} = \frac{r_0+1}{2^{r_0+2}} 
\label{mot1}
\ee
and at large $n$ limit, 
the expected number of horizontal steps at level 0 is 2:
\be
\lim_{n \to \infty} \frac{\sum_{r_0} r_0 M(r_0;n)}{M_n} = 2 \,.
\label{mot2}
\ee
\label{thm:mot}
\end{thm}

We plot the asymptotic distribution of \eqref{mot1}
as a function of $r_0$
in Figure \ref{mplot}.
At the limit of large size,
the number of Motzkin paths
exponentially 
decreases as $r_0$ increases when $r_0 \geq 1$,
and
the half of them
is the ones with $r_0=0$ and $1$.

\begin{figure}[ht]
\centering
\includegraphics[width=0.5\textwidth]{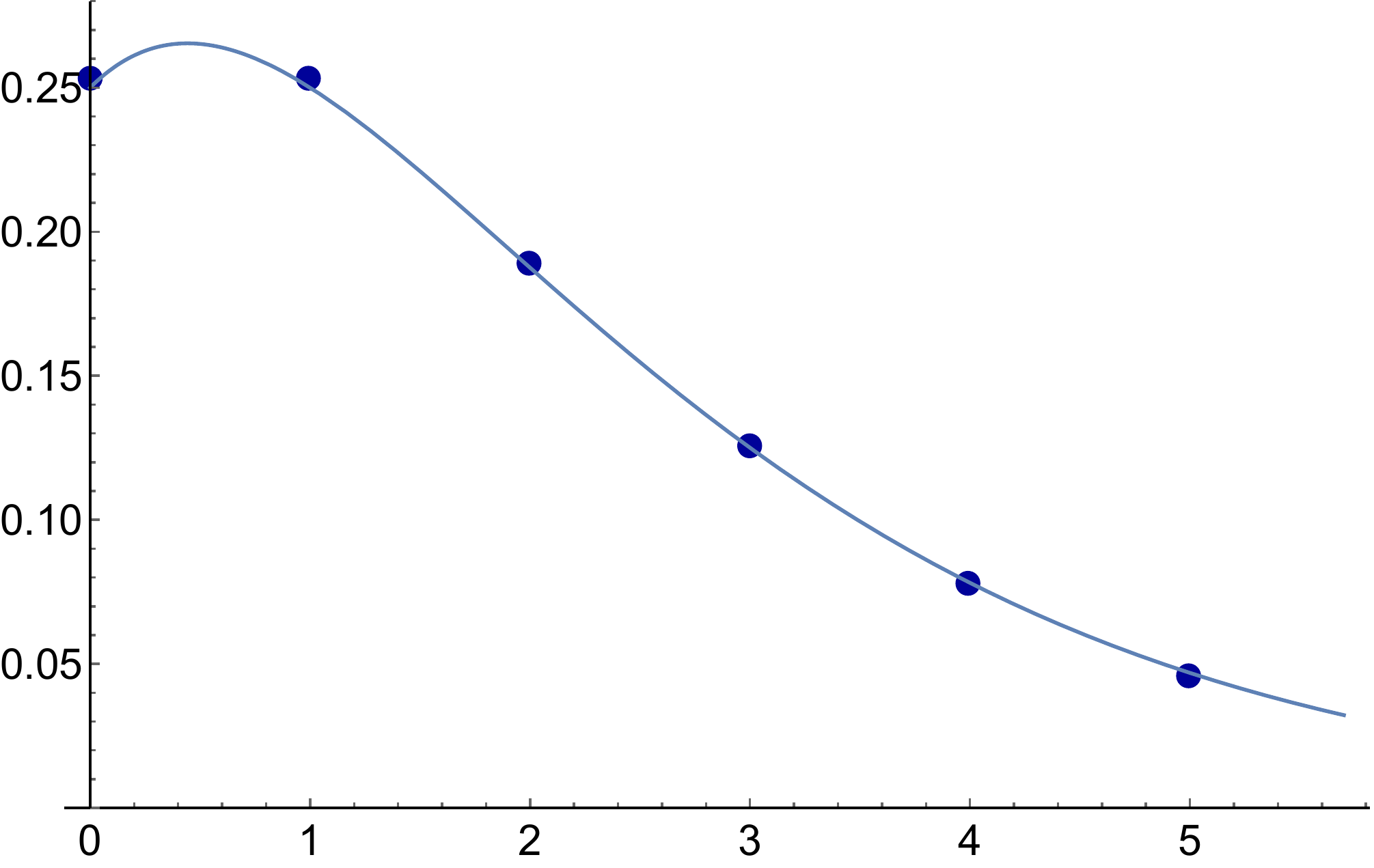}
\caption{Normalized distribution $M(r_0;n)/M_n$
of Motzkin paths as a function of $r_0$. The dots are given for $n=100$. 
The line is the asymptotic distribution as $n \to \infty$.}
\label{mplot}
\end{figure}

\subsection{$\pi$-shapes compatible with secondary structures}
In the previous subsection, we considered
Motzkin paths of size $n$, or equivalently, 
$\pi$-shapes of length $2n+2$.
On the other hand, 
instead of $\pi$-shapes with their own length fixed,
one may regard $\pi$-shapes
to which secondary structures of a given length are reducible \cite{LPC_2008}.
Namely, in this subsection,
we consider the class of $\pi$-shapes
compatible with secondary structures of length $\nu$,
and classify it according to the number of components.
The class of compatible $\pi$-shapes
is merely the collection of $\pi$-shapes
of length less than or equal to $\nu$.
However, we shall impose an additional constraint on it
so-called minimum arc-length
which was introduced to reflect the rigidity of the backbone of RNA
\cite{SW_1979}.
The minimum arc-length $\la$
indicates the condition on secondary structures
that each hairpin loop
consists of at least $\lambda-1$ unpaired vertices.
Therefore, 
we first construct $\pi$-shapes with $\la-1$ unpaired vertices
assigned in each hairpin loop
and collect such $\pi$-shapes
of length less than or equal to $\nu$.

Let $\pi_{\la}(r_0;\nu)$ denotes
the number of $\pi$-shapes compatible
with secondary structures of minimum arc-length $\la$,
components $r_0+1$ and length $\nu$,
and we define $\pi_{\la}(\nu) := \sum_{r_0}\pi_{\la}(r_0;\nu)$.
One may obtain its generating function 
by simply manipulating the expansion variables
of $m(t;v,w)$ as follows:
first, since $h=n-u+1$ from Table \ref{tab:1},
$x$ is the expansion parameter for $h$
by the change of variables $v \to v/x$ and $w \to x w$
with the multiplication of the overall factor $x$.
Second, the change of variables $x \to x z^{\la-1}$
and $w \to w z^2$ with the overall multiplication $z^2$
leads to the expansion parameter $z$ to count
the number of vertices including unpaired ones in hairpin loops.
Third, since only the number of vertices and components
are of interest here, 
taking irrelevant variables as unity, $x=v=w=1$ 
results in 
$z^{\la+1} m(t; z^{1-\la},z^{\la+1})$
which enumerates
the number of $\pi$-shapes with $\la-1$
unpaired vertices in each hairpin loop.
Finally, multiplying $1/(1-z)$,
one finds the generating function of $\pi_{\la}(r_0;\nu)$
\be
\sum_{r_0,\nu}\pi_{\la}(r_0;\nu) \, t^{r_0} z^{\nu}=
\frac{z^{\la+1}}{1-z}\, m(t;z^{1-\la},z^{\la+1}) \,.
\label{GFpi}
\ee
Using the generating function, one can 
evaluate the following asymptotics.
Its proof is given in Appendix \ref{app:2-2}.
\begin{thm} 
Let $\la$ be a positive integer.
Suppose that $\zeta_\la$ 
is the smallest root among
the positive real roots of 
the polynomial $z^{2\la+2}-4z^{\la+3}-2z^{\la+1}+1$.
For $r_0 \ll \nu$,  the asymptotic distribution of 
compatible $\pi$-shapes $\pi_{\la}(\nu)$
with respect to the number $r_0+1$ of components
is given as
\be
\lim_{\nu \to \infty}
\frac{\pi_{\la}(r_0;\nu)}{\pi_{\la}(\nu)} 
=(r_0+1) \biggl( \frac{\zeta_\la^2}{1+\zeta_\la^2} \biggr)
\biggl( \frac{1+\zeta_\la^{\la+1}}{2\big(1+\zeta_\la^2 \big)} \biggr)^{r_0}
\label{piasymp}
\ee
and the expected number of $r_0$ satisfies
\be
\lim_{\nu \to \infty} 
\frac{\sum_{r_0} r_0 \pi_{\la}(r_0;\nu)}{\pi_{\la}(\nu)}
= \frac{1-\zeta_\la^{\la+1}}{\zeta_\la^2} \,.
\ee
\label{thm:com}
\end{thm}
The expected number of components
in the collection of compatible $\pi$-shapes
is then $(1-\zeta_\la^{\la+1})/\zeta_\la^2 +1$ at large $\nu$ limit.
The case of $\la =4$, that is, the minimum of 3 unpaired vertices 
in each hairpin is often regarded as the most realistic one.
As an example, therefore, 
we plot the distribution for $\la=4$ in Figure \ref{pi4plot},
in which case, $\zeta_4 \approx 0.7563$
and the asymptotic distribution of \eqref{piasymp}
is given by $(r_0+1) \times 0.3639 \times 0.3968^{r_0}$.
The number of compatible $\pi$-shapes
exponentially decreases as the number of components increases.
The expected number of $r_0$ is
approximately $1.316$
and hence
the expected number of components
is approximately $2.316$.

\begin{figure}[ht]
\centering
\includegraphics[width=0.5\textwidth]{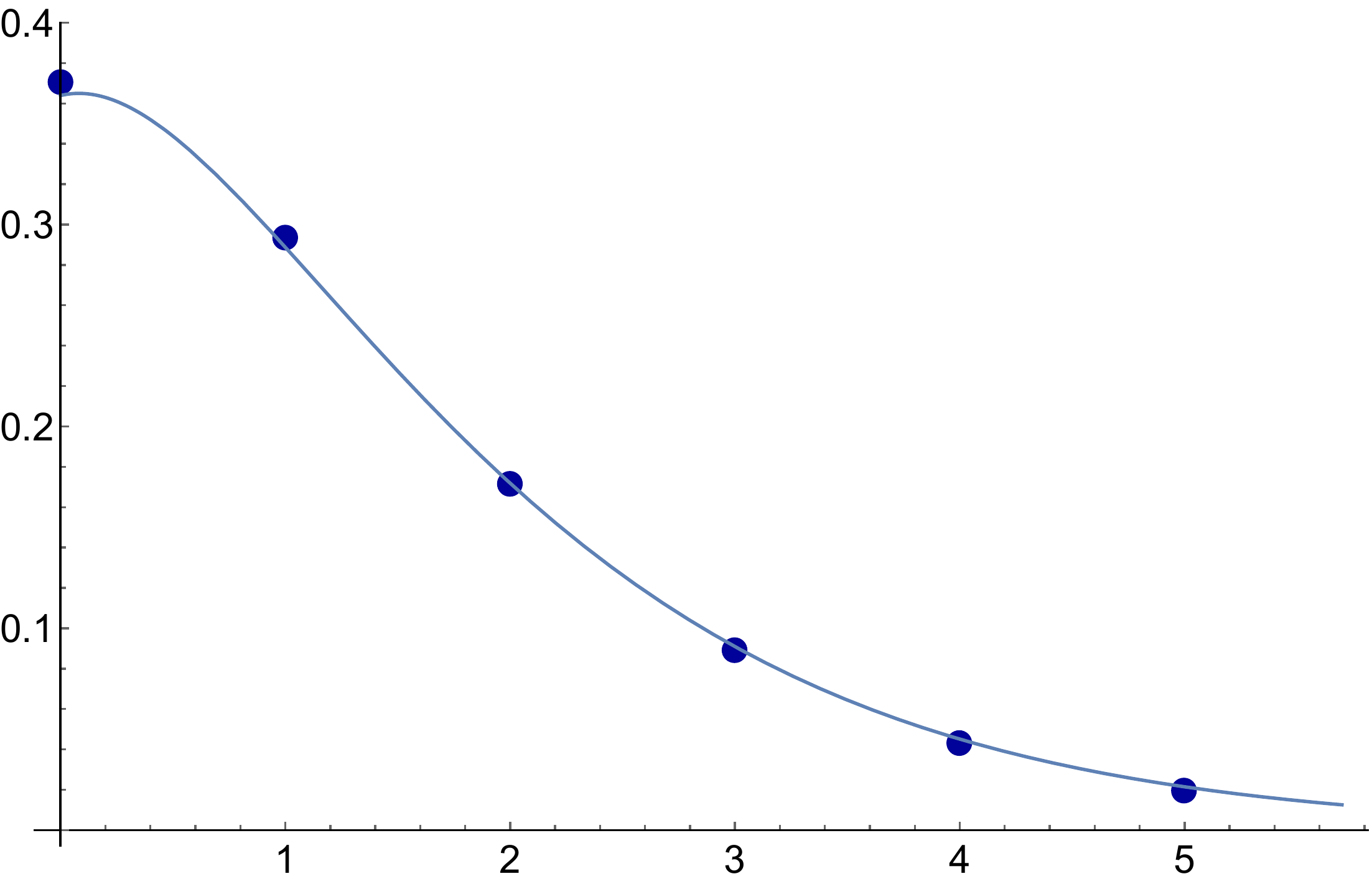}
\caption{Normalized distribution $\pi_4(r_0;\nu)/\pi_4(\nu)$
as a function of $r_0$. The dots are given for $\nu=200$. 
The line is the asymptotic distribution at the limit of $\nu \to \infty$.}
\label{pi4plot}
\end{figure}

\section{Conclusion}
In this paper, we established
the combinatorial interpretation of 2-Motzkin paths
to describe the island diagrams.
The generating function of island diagrams
was calculated using both 2-Motzkin paths
and Narayana numbers,
from which we found the identity \eqref{ouriden}
generalizing the Coker's identities.
The correspondence between
2-Motzkin paths and strings of matching brackets
reduces to the bijection map
between 1-Motzkin paths and $\pi$-shapes.
Subsequently, we classified
$\pi$-shapes by the number of components
and calculated 
the asymptotic distributions
and the expected number of components.
We observed that
the number of compatible $\pi$-shapes
exponentially decreases as the number of components increases.
In the case of $\la=4$, for instance,
the asymptotic distribution is 
given by $(r_0+1) \times 0.3639 \times 0.3968^{r_0}$
and the expected number of components
is approximately 2.316.

The bijection map 
between Motzkin paths and $\pi$-shapes
provides an additional combinatorial tool
to explore classifications and abstractions of secondary structures.
As one intuitive and immediate attempt 
concerning branching pattern similarities,
the number of components was taken into account.
We will furthermore study on
possible abstractions of RNA shapes,
which may greatly reduce the number of structures
while retaining structural similarities to some extent.

\vskip 10pt

\begin{center}
{\bf Acknowledgements}
\end{center} 
This work was supported by
the National Science Foundation
of China (Grant No. 11575119).

\begin{appendix}

\section{Proof of Proposition \ref{prop:mot}}
\label{app:1}
In order to classify Motzkin paths
according to the number of horizontal steps at level 0,
one may begin with Dyck paths.
A Dyck path is a lattice path
that never passes below the $x$-axis
and allows only the up step and the down step,
in other words,  
a Motzkin path without a horizontal step.
The Catalan number $C_u$
is the number of Dyck paths of length $2u$({\it i.e.}, of $u$ up steps),
for instance, the number $C_3=5$ of paths is shown in Figure \ref{fig:catalan}.
Furthermore, 
it is also known that
the number of Dyck paths
composed of the number $p$ of irreducible paths,
where by irreducible paths,
we mean Dyck paths that never touch the $x$-axis
between the starting and ending point.
The number is often referred to as Catalan $p$-fold convolution formula
\cite{C_1887, L_2000}:
\be
C(u;p)=\frac{p}u \binom{2u-p-1}{u-1} \,.
\ee
For example, $C(3,1)=2$, $C(3,2)=2$ and $C(3,3)=1$.

\begin{figure}[h]
\centering
\includegraphics[width=1.0\textwidth]{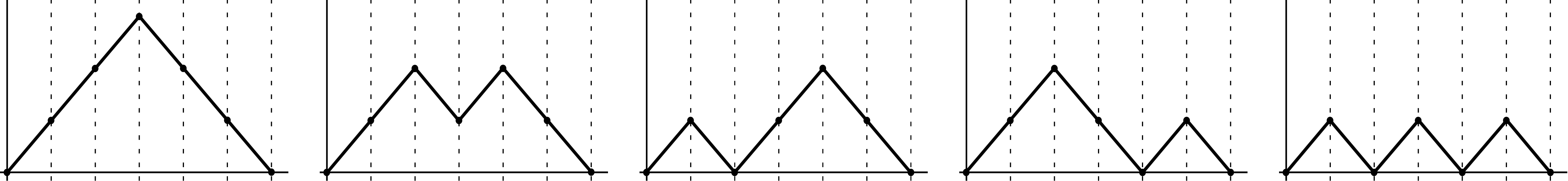}
\caption{\label{fig:catalan}
Dyck paths of $C_3$. From the left side,
the first two are of $C(3,1)$,
the next two are of $C(3,2)$
and the last one is of $C(3,3)$.}
\end{figure}

For given Dyck paths of $C(u;p)$,
one adds $n-2u$ horizontal steps
to obtain Motzkin paths of length $n$.
There are $2u+1$ places
into which one possibly inserts horizontal steps,
and of which $p+1$ places are at $x$-axis.
We first put $r_0$ horizontal steps into the $p+1$ places
and then put the remains into other places.
Summing over $p$, one finds
\be
M(r_0;n,u)= \sum_{p=1}^u \binom{r_0+p}{r_0}
\binom{n-r_0-p-1}{n-2u-r_0} C(u;p) \,.
\label{b2}
\ee
Note here that when $p=0$ or equivalently $u=0$,
we still have one Motzkin path consisting of
only horizontal steps and thus, $M(n;n,0)=1$.
One may manipulate
\eqref{b2} with the help of the identity analogous to 
Chu-Vandermonde's identity,
\be
\binom{m+n+1}{n}=\sum_{\a=0}^{n}
\binom{m+n-(t+\a)}{n-\a}\binom{t+\a}{\a}
\ee
where $0 \leq t \leq m$
and arrive at the final form
\be
M(r_0;n,u)=\frac{r_0+1}{n+1}\binom{n+1}u
\binom{n-r_0-u-1}{u-1} \,.
\ee

\section{Asymptotics}
\label{app:2}
In this section, we briefly demonstrate
the procedure calculating the asymptotics
needed in Theorem \ref{thm:mot} and Theorem \ref{thm:com}.
We shall first present theorems exploited to obtain
the asymptotics without proof.
For their proof and also
for more rigorous and complete discussions,
one can refer to the book by Flajolet and Sedgewick \cite{FS_2009}
and the paper by Flajolet and Odlyzko  \cite{FO_1990}.






\begin{thm}
{\rm (Exponential Growth Formula).}
If $f(z)$ is analytic at 0 and $R$
is the modulus of a singularity nearest to the origin
in the sense that
\be
R:=\mathrm{sup}\,\big\{r \geq 0 \,\big|\, f ~ is~ analytic~ in~ |z| <r \big\}\,,
\ee
then the coefficient $f_n=[z^n] f(z)$ satisfies
\be
f_n = R^{-n} \,  \theta(n)
\ee
where $\theta$ is a subexponential factor:
\be
\mathrm{lim \, sup}\, |\theta(n)|^{1/n} =1 \,.
\ee
\label{expgrowth}
\end{thm}
When $f_n = R^{-n} \,  \theta(n)$,
we say that $f_n$ is of exponential order $R^{-n}$.
The singularities which lie on the boundary of the disc of
convergence are called dominant singularities.
The following theorem
guarantees that counting generating functions
have positive real numbers 
as their dominant singularities.

\begin{thm}
{\rm (Pringsheim's theorem).}
If $f(z)$ is representable at the origin by
a series expansion that has non-negative coefficients
and radius of convergence $R$, then
the point $z=R$ is a singularity of $f(z)$.
\label{Pring}
\end{thm}

If a generating function $f(z)$ has its dominant singularity
at $z= \zeta$, the rescaled function $f(\zeta z)$
is analytic within the disc $|z| < 1$.
Then, one may apply the following theorem
to obtain the most dominant contribution
of the asymptotic.

\begin{defn}
Given two numbers $\phi$, $R$
with $R>1$ and $0 < \phi < \frac{\pi}2$,
the open domain $\Delta_1(\phi,R)$ is defined as
(Fig. \ref{fig:domain})
\be
\Delta_1(\phi,R)=\,\big\{z \,\big|\, |z| < R \,, 
~ z \neq 1 \,, ~ |\mathrm{arg}(z-1)| >\phi \big\} \,.
\ee
A function is $\Delta_1$-analytic if it is
analytic in $\Delta_1(\phi,R)$ for some $\phi$ and $R$.
\end{defn}

\begin{thm}
{\rm (Flajolet and Odlyzko).}
Assume that $f(z)$ is $\Delta_1$-analytic and
\be
f(z) \sim (1-z)^{-\a}  ~~~~ as ~ z \to 1\,, ~~z \in \Delta_1
\ee
with $\a \notin \{0,-1,-2, \cdots\}$.
Then, the coefficients of $f$ satisfy
\be
[z^n] f(z) \sim \frac{n^{\a-1}}{\Gamma(\a)} \,.
\ee
\label{thmFO}
\end{thm}

\begin{figure}[h]
\centering
\includegraphics[width=0.7\textwidth]{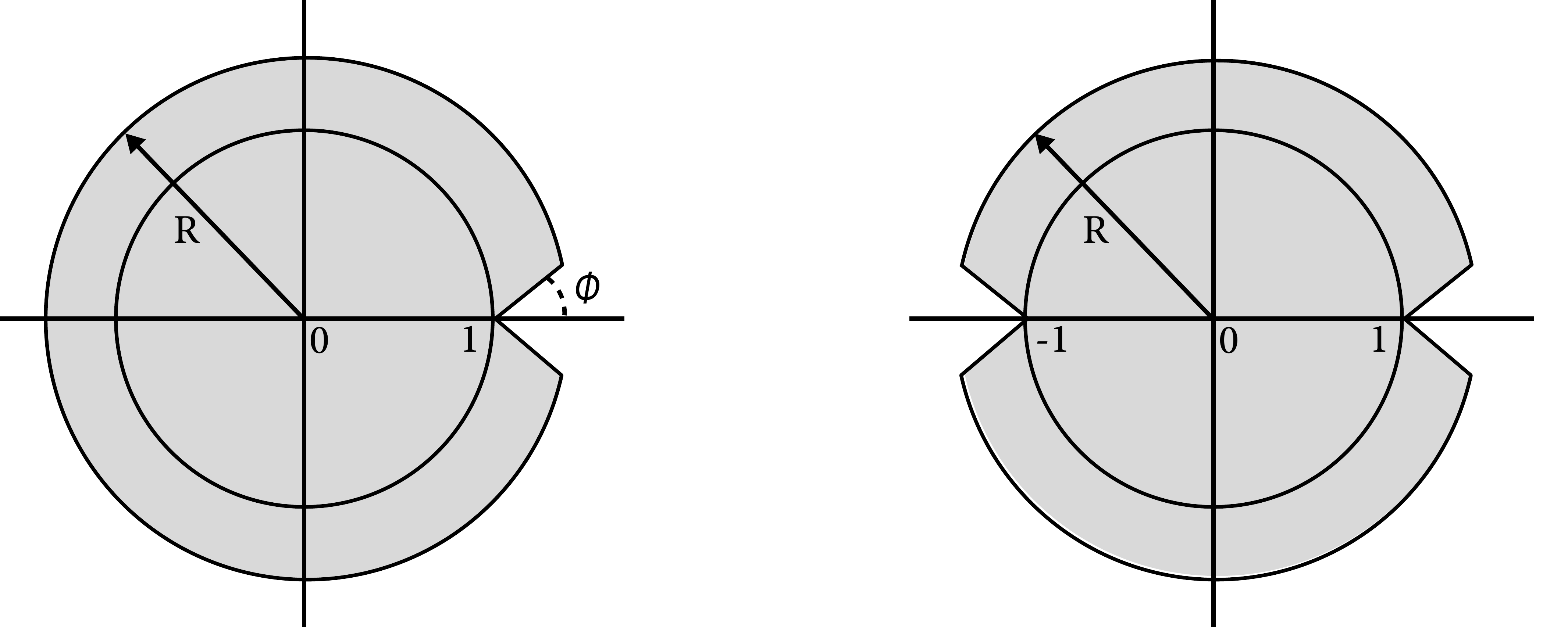}
\caption{The shaded region is $\Delta_1$ domain (left).
When there are two dominant singularities
at $z=\pm 1$, the generating function is required
to be analytic in the domain on the right.}
\label{fig:domain}
\end{figure}

\subsection{Proof of Theorem \ref{thm:mot}}
\label{app:2-1}

We need to find the asymptotics
of $M_n$, $M(r_0;n)$ and $\sum_{r_0} r_0 M(r_0;n)$.
Among the three,
the asymptotic of $M_n$ is already well-known:
$M_n \sim {3^{n+\frac32}}/({2 \sqrt{\pi} \, n^{3/2}})$.
We will not derive it here, nevertheless,
one may easily find it as well using the method given below.
Then, let us first consider $\sum_{r_0} r_0 M(r_0;n)$.
From \eqref{cMN}, one knows
\be
\sum_{r_0} r_0 M(r_0;n)
= [w^n]\, t \frac{\partial}{\partial t} m(t;1,w) \, \bigg|_{t=1}
= [w^n]\, w \, m(1,w)^2
\ee
where the generating function is explicitly given as
\be
w \, m(1,w)^2=-\frac{w^2+2w-1}{2w^3}
-\frac{(1-w)\sqrt{(1+w)(1-3w)}}{2w^3} \,.
\label{wm^2}
\ee

The former term in the right hand side of \eqref{wm^2}
is just to cancel singular terms in the expansion of $w$
and has no contributions in the coefficient of $w^n$ as $n \to \infty$.
Thus, we consider only the latter part in the singular expansion below.
Since the dominant singularity is at $\zeta := 1/3$,
by the rescaling $w \to \zeta w$,
the function $\zeta w \, m(1, \zeta w)^2$ is $\Delta_1$-analytic.
Around the point $1$, the expansion of the rescaled function
is given by
\be
-6 \sqrt3 \, (1-w)^{1/2}-\frac{81}4 \sqrt3 \, (1-w)^{3/2}
+\cO\big((1-w)^{5/2}\big) \,.
\ee
Then Theorem \ref{thmFO} applies to yield\footnote{
In order to determine the next order contribution,
one should take account of the subdominant
contribution coming from the term of $(1-w)^{1/2}$ 
as well as the dominant one from the term of $(1-w)^{3/2}$ .
Refer to Theorem VI.1 on page 381 of \cite{FS_2009}
}
\be
[w^n] \zeta w \, m(1, \zeta w)^2 
\sim -6 \sqrt3  \, \frac{n^{-3/2}}{\Gamma(-\frac12)}
+\cO(n^{-5/2}) \,.
\ee
Noting that 
$[w^n] w\,m(1,w)^2 = \zeta^{-n} [w^n] \zeta w \, m(1, \zeta w)^2$,
one obtains the dominant contribution in the asymptotic expression as
\be
\sum_{r_0} r_0 M(r_0;n) = [w^n] w \, m(1,w)^2
\sim \frac{3^{n+\frac32}}{\sqrt{\pi} \, n^{3/2}} \,.
\label{1asym1}
\ee
Therefore, with the asymptotic of Motzkin number,
we prove \eqref{mot2} in Theorem \ref{thm:mot}
that the expected number of $r_0$ is 2 at large $n$ limit.
In brief, 
we summarize our procedure
to find the asymptotics as follows:
given a generating function $f(z)$
with its dominant singularity being at $z=\zeta$,
if $f(\zeta z)$ is $\Delta_1$-analytic
and is written as $f(\zeta z)= Q(\zeta z)(1-z)^{-\a}$
with $\a \notin \{0,-1,\cdots \}$,
then $[z^n] f(z) \sim \zeta^{-n} Q(\zeta) \frac{n^{\a-1}}{\Gamma(\a)}$.

Let us now apply the procedure
to demonstrate the asymptotic distribution
\eqref{mot1} in Theorem \ref{thm:mot}.
From \eqref{cMN}, one knows
\be
M(r_0;n) = [w^n]\, w^{r_0} \, A(1,w)^{r_0+1} \,.
\ee
We evaluate the asymptotic 
of the coefficient of $w^n$ as $n \to \infty$ for some fixed $r_0$,
and thus we assume $r_0 \ll n$.
Putting $r_0 = k-1$ for notational simplicity,
the function is
explicitly written as
\be
w^{k-1} \, A(1,w)^{k}
=A_f^k (w)+ A_s^k (w)
\ee
where
\be
A_f^k(w)=\frac1{2^k w} \sum_{b=0}^{\lfloor \frac{k}2 \rfloor}
\binom{k}{2b} \left(\frac{1-3w}{1+w} \right)^b
 \,, ~~~ A_s^k(w) = -\frac1{2^k w} 
 \sum_{b=0}^{\lfloor \frac{k-1}2 \rfloor}
\binom{k}{2b+1} \left(\frac{1-3w}{1+w} \right)^{b+\frac12} \,.
\ee
Since $w \, A_f^k(w)$ is analytic at $0$
with the dominant singularity at $w= -1$,
by Theorem \ref{expgrowth},
the coefficient of $w^n$ in $A_f^k(w)$ 
is of exponential order $(-1)^{-n+1}$
whereas
the one in $A_s^k(w)$ 
is of exponential order $3^n$.
Therefore, for $k \ll n$,
the contributions of $A_f^k(w)$ 
in the coefficients of $w^n$ is negligible.
Furthermore, regarding Theorem \ref{thmFO},
one may anticipate that the most dominant 
contribution in $A_s^k(w)$ 
comes from the term of $b=0$.
Thus,
as far as the most dominant asymptotic is only concerned,
we can simply apply the procedure given above to the function,
\be
f(w):=-\frac{k}{2^k}\sqrt{\frac1{1+w}} \, (1-3w)^{1/2} \,.
\ee
Observing that $[w^{n+1}]f(w) \sim [w^n] A_s^k(w)$,
one arrives at
\be
[w^n]\,w^{k-1} A(1,w)^k  \sim 3^{n+1} 
\bigg(-\frac{k}{2^k} \sqrt{\frac{3}{4}} \, \, \bigg) \,
\frac{n^{-3/2}}{\Gamma(-\frac12)}
\ee
and therefore,
\be
M(r_0;n) \sim  
\frac{(r_0+1) \, 3^{n+\frac32}}{2^{r_0+3} \sqrt\pi \, n^{3/2}} \,.
\label{1asym2}
\ee
The asymptotics in \eqref{1asym1} and \eqref{1asym2}
together with $M_n \sim {3^{n+\frac32}}/({2 \sqrt{\pi} \, n^{3/2}})$,
one obtains the formulae in Theorem \ref{thm:mot}.

\subsection{Proof of Theorem \ref{thm:com}} 
\label{app:2-2}
One can basically follow the procedure given above
to find the asymptotics in Theorem \ref{thm:com},
although there is one more thing to consider:
multiple dominant singularities.
Let us first consider $\pi_\la(\nu)$.
From the equation \eqref{GFpi},
\be
\pi_\la(\nu) = [z^\nu] \, \frac{z^{\la+1}}{1-z}\, m(z^{1-\la},z^{\la+1}) 
\ee
where
\be
\frac{z^{\la+1}}{1-z}\, m(z^{1-\la},z^{\la+1})
=\frac{1-z^{\la+1}-\sqrt{z^{2\la+2}-4z^{\la+3}-2z^{\la+1}+1}}
{2z^2(1-z)} \,.
\ee
We write the polynomial in the square root as
\be
z^{2\la+2}-4z^{\la+3}-2z^{\la+1}+1 = Q_{\la}(z) \, (1-z/\zeta_\la)
\ee
where $\zeta_\la$ is the dominant singularity.
The fact that $\zeta_\la$ is a positive real number
is guaranteed by Pringsheim's theorem (Theorem \ref{Pring}).
Following the procedure in the proof of Theorem \ref{thm:mot},
one may simply obtain
\be
\pi_\la(\nu) \sim \left(\frac1{\zeta_\la}\right)^{\nu+2} 
\left(-\frac{\sqrt{Q_\la(\zeta_\la)}}{2(1-\zeta_\la)}\right) \,
\frac{\nu^{-3/2}}{\Gamma(-\frac12)} \,.
\label{2asym1}
\ee
Note that, however, this formula holds for even numbers of $\la$.

When $\la$ is an odd number,
we have two dominant singularities,
which are $+\zeta_\la$ and $-\zeta_\la$.
The generating function is then analytic 
in the domain depicted in Figure \ref{fig:domain} after the rescaling.
In this case,
the two contributions 
from each of the singularities
are added up 
to give the asymptotic
(for rigorous and complete arguments,
refer to Theorem VI.5 on page 398 of \cite{FS_2009}).
The presence of 
the two dominant singularities
reflects the fact
that the coefficient 
of $z^\nu$ with the odd $\nu=2k+1$ 
equals to the one with the even $\nu=2k$:
An odd number of $\la$ means
assigning an even number of vertices in each hairpin
and hence,
there is no compatible $\pi$-shape with $\nu$ odd.
This feature can be shown from the explicit calculation.
Let us write the polynomial for $\la$ odd as
\be
z^{2\la+2}-4z^{\la+3}-2z^{\la+1}+1 
= R_{\la}(z) \, (1-z/\zeta_\la) \, (1+z/\zeta_\la) \,.
\ee
Then the two contributions are added up and give
\be
\pi_\la(\nu) \sim \left(\frac1{\zeta_\la}\right)^{\nu+2} 
\left(-\frac{\sqrt{2\, R_\la(\zeta_\la)}}{2(1-\zeta_\la)}\right) \,
\frac{\nu^{-3/2}}{\Gamma(-\frac12)}
+ \left(\frac1{-\zeta_\la}\right)^{\nu+2} 
\left(-\frac{\sqrt{2\, R_\la(\zeta_\la)}}{2(1+\zeta_\la)}\right) \,
\frac{\nu^{-3/2}}{\Gamma(-\frac12)}\,.
\label{2asym2}
\ee
For both the even $\nu=2k$ and the odd $\nu=2k+1$, one finds
\be
\pi_\la(2k)=\pi_\la(2k+1) \sim 
\left(\frac1{\zeta_\la}\right)^{2k+2} 
\left(-\frac{\sqrt{2\, R_\la(\zeta_\la)}}{(1-\zeta_\la)^2}\right) \,
\frac{(2k)^{-3/2}}{\Gamma(-\frac12)} \,.
\ee

Let us now consider the asymptotic of $\sum_{r_0} r_0 \pi_\la(r_0; \nu)$.
Its generating function is given as
\be
\sum_{r_0} r_0 \pi_\la(r_0; \nu) = 
[z^\nu]  \, 
\frac{z^{\la+1}}{1-z}\, 
\frac{\partial}{\partial t} \, m(t;z^{1-\la},z^{\la+1}) \bigg|_{t=1}
=[z^\nu] \, \frac{z^{2\la+2}}{1-z} \, 
\left( m(z^{1-\la},z^{\la+1}) \right)^2 \,.
\ee
The significant part of the generating function is
\be
-\frac{(1-z^{\la+1}) \sqrt{z^{2\la+2}-4z^{\la+3}-2z^{\la+1}+1 }}
{2z^4 \,(1-z)} \,.
\ee
It is now straightforward to find its asymptotic.
For an even number of $\la$,
\be
\sum_{r_0} r_0 \pi_\la(r_0; \nu) \sim
\left(\frac1{\zeta_\la}\right)^{\nu+4} 
\left(-\frac{(1-\zeta_\la^{\la+1}) \, \sqrt{Q_\la(\zeta_\la)}}
{2(1-\zeta_\la)}\right) \,
\frac{\nu^{-3/2}}{\Gamma(-\frac12)}
\label{2asym3}
\ee
and for an odd number of $\la$,
\be
\sum_{r_0} r_0 \pi_\la(r_0; \nu) \sim
\left(\frac1{\zeta_\la}\right)^{\nu+4} 
\left( \frac1{1-\zeta_\la}+\frac{(-1)^\nu}{1+\zeta_\la} \right)
\left(-\frac{(1-\zeta_\la^{\la+1}) \, \sqrt{2 R_\la(\zeta_\la)}}
{2}\right) \,
\frac{\nu^{-3/2}}{\Gamma(-\frac12)} \,.
\label{2asym4}
\ee

In the same manner, 
one may easily find the asymptotic of $\pi_\la(r_0; \nu)$.
From \eqref{GFpi} and \eqref{cMN}, we know
\be
\pi_\la(r_0; \nu)=[z^\nu] \,
\frac1{1-z} \left( z^{\la+1} \, A(z^{1-\la},z^{\la+1}) \right)^{r_0+1} \,.
\ee
The most dominant contribution comes from the term,
\be
-\frac{(r_0+1)(1+z^{\la+1})^{r_0} 
\sqrt{z^{2\la+2}-4z^{\la+3}-2z^{\la+1}+1}}
{2^{r_0+1} (1-z)(1+z^2)^{r_0+1}} \,.
\ee
Therefore, for an even number of $\la$,
\be
\pi_\la(r_0; \nu) \sim
\left(\frac1{\zeta_\la}\right)^{\nu} 
\left(-\frac{(r_0+1)(1+\zeta_\la^{\la+1})^{r_0} \, \sqrt{Q_\la(\zeta_\la)}}
{2^{r_0+1}(1-\zeta_\la)(1+\zeta_\la^2)^{r_0+1}}\right) \,
\frac{\nu^{-3/2}}{\Gamma(-\frac12)}
\label{2asym5}
\ee
and for an odd number of $\la$,
\be
\pi_\la(r_0; \nu) \sim
\left(\frac1{\zeta_\la}\right)^{\nu} 
\left( \frac1{1-\zeta_\la}+\frac{(-1)^\nu}{1+\zeta_\la} \right)
\left(-\frac{(r_0+1)(1+\zeta_\la^{\la+1})^{r_0} \, \sqrt{2R_\la(\zeta_\la)}}
{2^{r_0+1}(1+\zeta_\la^2)^{r_0+1}}\right) \,
\frac{\nu^{-3/2}}{\Gamma(-\frac12)} \,.
\label{2asym6}
\ee

The asymptotics given in \eqref{2asym1}, \eqref{2asym2}, \eqref{2asym3},
\eqref{2asym4}, \eqref{2asym5} and \eqref{2asym6}
prove Theorem \ref{thm:com}.

\end{appendix}


\end{document}